\documentclass[11pt,oneside]{article}

\usepackage{graphicx}
\usepackage{amsmath}
\usepackage{amsfonts}
\usepackage{amssymb}
\usepackage{amsthm}
\usepackage{enumerate}
\usepackage{datetime}
\usepackage{url}
\usepackage{times} 
\usepackage[round,colon,authoryear]{natbib} 

\newcommand{\R}{\mathbb{R}}
\newcommand{\E}{\mathbb{E}}
\newcommand{\N}{\mathbb{N}}
\newcommand{\1}{\mathbf{1}}
\newcommand{\Prob}{\mathbb{P}}
\newcommand{\Q}{\mathbb{Q}}
\newcommand{\CF}{\mathcal{F}}

\newcommand{\dd}{\mathrm{d}}

\newtheorem{thm}{Theorem}
\newtheorem{prop}{Proposition}
\newtheorem{lemma}{Lemma}
\newtheorem{cor}{Corollary}

\theoremstyle{remark}
\newtheorem{ex}{Example}
\theoremstyle{remark}
\newtheorem{defin}{Definition}
\theoremstyle{remark}

\theoremstyle{remark}
\newtheorem{remark}{Remark}
\theoremstyle{remark}

\allowdisplaybreaks

\title{A new proof for the conditions of Novikov and Kazamaki\thanks{I am very grateful to Ioannis Karatzas for his encouragement  to pursue this project. I thank Sara Biagini, Tomoyuki Ichiba, Kostas Kardaras, Martin Larsson, Gechung Liang, Marcel Nutz, Sergio Pulido, and Mathieu Rosenbaum for many helpful discussions on the subject matter. I am indebted to an anonymous referee for her or his careful comments, which lead to a substantial improvement of this paper. This paper is dedicated to Ioannis Karatzas in honor of his 60th birthday.}}

\author{Johannes Ruf\footnote{E-mail: johannes.ruf@oxford-man.ox.ac.uk} \\
    Oxford-Man Institute of Quantitative Finance and Mathematical Institute\\University of Oxford\\ 
    \today 
    }
\date{}

\begin{document}
\thispagestyle{plain} \maketitle
\begin{abstract}
    This paper provides a novel proof for the sufficiency of certain well-known criteria that
    guarantee the martingale property of a continuous, nonnegative local martingale. More precisely, it is shown that
    generalizations of Novikov's condition and Kazamaki's criterion follow directly from
    the existence of F\"ollmer's measure. This approach allows to extend well-known criteria of martingality from strictly positive to only nonnegative, continuous local martingales.
    
\end{abstract}
\noindent

\section{Introduction}
Fix a continuous, nonnegative local martingale $Z^L$ of the form $Z^L = \mathcal{E}(L):= \exp(L - \langle L\rangle/2)$ on some filtered probability space $(\Omega, \mathbb{F}, \{\CF_t\}_{t \geq 0}, \Prob)$. Here, $L$ denotes  another continuous local martingale on $[0, T_0)$ with $L_0 = 0$,
where $T_0$ is the first hitting time of zero by $Z^L$. The stopping time $T_0$ is also the first hitting time of infinity by the quadratic variation $\langle L \rangle$, as we will demonstrate below.  We refer the reader to Subsection~\ref{SS prelims} for the precise definition of a local martingale on a stochastic interval $[0, T_0)$. 

We are interested in establishing sufficient conditions that guarantee that $Z^L$ is a
(uniformly integrable) martingale, namely, that $Z^L$ satisfies $Z^L_t = \E[Z^L_T|\CF_t]$ for some fixed time horizon $T \in [0,\infty]$.
Towards this end, let $\mathcal{T}$ denote
the set of all stopping times $\tau$ for which there exists some $n_{\tau} \in \N$ with $\tau \leq (T-1/n_{\tau}) \wedge n_{\tau}$.  Then, in Section~\ref{S proof}, we shall prove the following result:
\begin{thm}[Abstract version of the Novikov-Kazamaki conditions]  \label{C abstract}
    Let $f: \R \times [0,\infty) \rightarrow [0,\infty)$ denote a continuous function such that 
        \begin{align*} 
            \limsup_{t \uparrow \infty} f(B_t+t,t) \cdot \exp\left(-B_t - \frac{t}{2}\right) = \infty
        \end{align*}
    almost surely for some (and thus, for any) Brownian motion $B$. If
    \begin{align}   \label{E cond0}
        \sup_{\tau \in \mathcal{T}} \left\{\E\left[f(L_\tau, \langle L \rangle_\tau)  \1_{\{Z^L_\tau > 0\}}\right]\right\}  < \infty,
    \end{align}    
    then $Z^L$ is a (uniformly integrable) martingale on $[0,T]$.
\end{thm}
Theorem~\ref{C abstract} applied to $f(x,y) = \exp(y/2)$ now directly implies the sufficiency of
	\begin{align*}
		\E\left[\exp\left(\frac{1}{2} \langle L \rangle_T\right)\right] < \infty
	\end{align*}
(\emph{Novikov's condition}) and, applied to $f(x,y) = \exp(x/2)$, the sufficiency of
	\begin{align}  \label{E Kaz criterion}
		\sup_{\tau \in \mathcal{T}} \E\left[\exp\left(\frac{1}{2} L_\tau\right)\right] < \infty
	\end{align}
(\emph{Kazamaki's criterion}) for the uniform integrability and martingale property of a strictly positive local martingale $Z^L$.  Both these criteria can be embedded in a large family of sufficient conditions that we shall study in Section~\ref{S further}.
To begin with, define, for any $a \in \R$ and any measurable function $\phi:[0,\infty)\rightarrow \R$, the process
    \begin{align*} 
        S^{L,a,\phi} := \exp\left(a L + \left(\frac{1}{2} - a\right) \langle L\rangle
        -|a-1| \phi(\langle L\rangle)\right)
        \1_{\{Z^L > 0\}}.
    \end{align*}

We shall need the concept of a \emph{lower function}, which we shall briefly review in the Appendix for the reader's convenience.
 For the present discussion, it is sufficient to note that a continuous function is a lower function if and only if $\limsup_{t \uparrow \infty} (B_t - \phi(t)) = \infty$ holds almost surely for some (and thus, for any) Brownian motion $B$. 

We now are  ready to formulate a generalized version of Novikov's condition and Kazamaki's criterion:
\begin{cor}[Novikov's condition and Kazamaki's criterion]  \label{t novikov}
	The stochastic exponential $Z^L = \mathcal{E}(L)$ is a (uniformly integrable) martingale on $[0,T]$ if, for some $a \in \R \setminus\{1\}$ and some continuous lower function $\phi$, we have
    \begin{align}  \label{E cond}
        \sup_{\tau \in \mathcal{T}} \left\{\E\left[S^{L,a,\phi}_{\tau} \right]\right\}  < \infty.
    \end{align}
   Applied to a strictly positive local martingale $Z^L$,  \eqref{E cond} with $a = 0$ and $\phi = 0$ implies Novikov's condition  (as $\langle L\rangle$ is increasing) and $a = 1/2$ and $\phi = 0$ implies Kazamaki's criterion.
\end{cor}
\begin{proof}
	Define $f: \R \times [0,\infty) \rightarrow [0,\infty)$ by $$f(x,y) = \exp\left(a x + \left(\frac{1}{2}-a\right)y - |a-1| \phi(y)\right)$$ for all $(x,y) \in \R \times [0,\infty)$ and let $B$ denote some Brownian motion. Then observe that
$S^{L,a,\phi} = f(L, \langle L \rangle) \1_{\{Z^L > 0\}}$ and that
        \begin{align*} 
            \limsup_{t \uparrow \infty} f(B_t+t,t) \cdot \exp\left(-B_t - \frac{t}{2}\right) = \limsup_{t \uparrow \infty} \exp\left((a-1) B_t - |a-1| \phi(t)\right) = \infty.
        \end{align*}
	Thus, an application of Theorem~\ref{C abstract} yields the statement.
\end{proof}

We emphasize that Corollary~\ref{t novikov} has been proven before, at least for strictly positive local martingales $Z^L$; we shall give an overview of the relevant literature below.
However, our proof is, to the best of our knowledge, new and seems to be shorter and simpler than the existing proofs. It relies on the existence of F\"ollmer's measure, as constructed in \citet{M}. Such a probability measure is defined for any local martingale, and, in particular, yields a necessary and sufficient condition for the martingale property of $Z^L$ in terms of an explosion of the quadratic variation $\langle L \rangle$ of $L$. With this condition, the theorem can easily be proved by contradiction. Indeed, \eqref{E cond0} guarantees that explosions of $\langle L \rangle$ cannot occur under 
F\"ollmer's measure.  



\subsection*{Review of extant literature}
We shall provide some pointers to the relevant literature  on local and true martingales. The following list is by no means close to being complete.

\citet{Girsanov_1960} posed the problem of deciding whether a stochastic exponential is a true martingale or not. \citet{Gikhman_Skorohod_1972} and \citet{Lip_Shir_1972} provided
sufficient conditions for the martingale property of a stochastic exponential. These conditions were then first generalized by \citet{Novikov} and later by
\citet{Kazamaki_1977}, who derived the cases $\phi \equiv 0$ and $a = 0$ or $a = 1/2$, respectively, in \eqref{E cond}.
\citet{Krylov_Novikov} provides a simple proof of these results.

\citet{Novikov_conditions_exp, Novikov_conditions_cont} observed that it is possible to
include lower functions in the criterion for the special cases $a = 0$ and $a =1/2$ under some Gaussian assumptions. This has been generalized to any continuous local martingale, again for the cases $a = 0$ and $a =1/2$, by \citet{ChernyShiryaev}.

\citet{Lepingle_Memin_integrabilite} showed the sufficiency of the uniform integrability of $\{S^{L,a,0}_{\tau} \}_{\tau \in \mathcal{T}}$ with $a \in [0,1)$ for the martingale property of $\mathcal{E}(L)$.
\citet{Okada_1982} extended this result by allowing lower functions of the form $\phi(t) = C\sqrt{t}$.  The most general result in the form of Corollary~\ref{t novikov}, for strictly
positive local martingales, has been provided by
\citet{Kazamaki_1983}.

If either the local martingale $L$ or $Z^L$ satisfies additional structural assumptions, then one can often give more precise sufficient, and possibly also necessary conditions. For example, if $L$ is a BMO martingale, then $Z^L$ is always a martingale, as shown in \citet{Kazamaki_1983}.  If $L$ is a stochastic integral of solutions to an SDE,  \citet{Engelbert_Schmidt_1984} and \citet{Stummer_1993}  discuss
the martingale property of $Z^L$; see also  \citet{MU_WH, MU_martingale} for a complete characterization of martingality in the one-dimensional case.
The question of martingality for a strongly Markovian process is treated in
\citet{DShir}, \citet{Kotani_2006}, \citet{Blei_Engelbert_2009}, and \citet{HP_visual}.
We refer to \citet{Mayerhofer_2011} and the references therein for necessary and sufficient conditions for $Z^L$ being a martingale if $L$ is an affine process.

We remark that the  case of a discontinuous local martingale $L$ has also been deeply studied. For an overview of the literature, we refer to \citet{Lepingle_Memin_Sur}, \citet{Kallsen_Shir}, \citet{CFY}, and \citet{Protter_Shimbo}. Finally, we note that
\citet{Elworthy_Li_Yor_97} provide a precise formula for the expectation of a continuous local martingale in terms of the tails of its quadratic variation. For further pointers to this literature, we refer to \citet{Rheinl_2010}.

\section{A new proof for Novikov- and Kazamaki-type conditions}\label{S proof}
In this section, we present the proof of Theorem~\ref{C abstract}. To the best of our knowledge, it is a novel argument, which is based on
the existence of a certain probability measure, constructed via an extension theorem applied to a consistent family of probability measures generated by stopped versions of the local martingale $Z^L$. With this tool at hand, the proof reduces to a very short argument.

\subsection{Extended stochastic exponential} \label{SS prelims}
In the spirit of Appendix~A in \citet{CFR2011}, we call a stochastic process $L$ a continuous local martingale on $[0,\tau)$ for some predictable positive stopping time $\tau>0$ if the stopped process $L^{\widetilde{\tau}}_\cdot := L_{\cdot \wedge \widetilde{\tau}} $ is a continuous local martingale for any stopping time $\widetilde{\tau}< \tau$. With $\tau = \infty$ we have the usual class of continuous local martingales.

\begin{lemma}[Extended stochastic exponential]  \label{L extended}
    Fix a predictable positive stopping time $\tau>0$ and a continuous local martingale $L$ on $[0,\tau)$ and consider the exponential local martingale
    $Z^L = \mathcal{E}(L)= \exp(L - \langle L\rangle/2)$  on $[0,\tau)$. Then the random variable $Z_{\tau}^L:= \lim_{t \uparrow \tau}  Z_t^L$ exists,
 is nonnegative, and satisfies $\{\lim_{n \uparrow \infty} \langle L\rangle_{\tau_n}  < \infty\} = \{Z_\tau^L > 0\}$ almost surely.
\end{lemma}
\begin{proof} 
Doob's downcrossing inequality yields that  $Z_{\tau(\omega)}^L(\omega)$ exists for almost all $\omega \in \Omega$; see the proof of Theorem~1.3.15 in \citet{KS1} with $\infty$ replaced by $\tau$ and $n$ replaced by $\tau_n$ for all $n \in \N$  for a nondecreasing sequence of stopping times $\{\tau_n\}_{n \in \N}$ with $\lim_{n \uparrow \infty} \tau_n = \tau$. Next, observe that $Z_\tau^L = 0$ if and only if $\log(Z_\tau^L) = -\infty$ and that
    \begin{align*}
        \log\left(Z_t^L\right) = \langle L \rangle_t\left(\frac{L_t}{\langle L \rangle_t} - \frac{1}{2}\right).
    \end{align*}
for all $t \in (0, \tau)$ with $\langle L \rangle_t>0$. Thus, to prove the statement it is sufficient to show that $\lim_{t \uparrow \tau} L_t/\langle L \rangle_t$ exists and is real. This, however, follows directly from an application of the Dambis-Dubins-Schwarz theorem; see also Exercise~V.1.16.3 and  Proposition~V.1.8 in \citet{RY}.
\end{proof}

Set $\mathcal{E}(L)_{\tau+t} := \mathcal{E}(L)_{\tau} := \lim_{s \uparrow \tau}  \mathcal{E}(L)_{s}$ for all $t \geq 0$ for a continuous local martingale $L$ on $[0, \tau)$. Then, for any nonnegative continuous local martingale $Z$, there exists a continuous local martingale $L$ on $[0, T_0)$, measurable with respect to the filtration generated by $Z$, such that  $Z = Z^L := \mathcal{E}(L)$, where $T_0$ denotes
the first hitting time of zero by  $Z$.
For a positive continuous local martingale $Z$, this is Proposition~VIII.1.6 in \citet{RY}.

\subsection{Change of measure for continuous local martingales}
In this subsection, we provide a generalization of Girsanov's theorem, proven in its modern version by \citet{vanschuppenWong}, to nonnegative local martingales. This generalization goes back to \citet{F1972}, who constructed a similar probability measure on the product space $\Omega \times [0,\infty]$, endowed with the predictable sigma-field, for a nonnegative supermartingale, such that its expectation can be represented as the probability of a certain event. \citet{M} observed that such a probability measure can already be constructed on certain spaces $(\Omega, \CF)$ if the supermartingale is a local martingale. We shall follow here this approach, which was then taken on by 
\citet{DS_Bessel}, \citet{PP}, \citet{FK}, \citet{Ruf_hedging}, and many others.  

We remark that most of the statements in Theorem~\ref{P change} have been proven before, for example in \citet{CFR2011}. However, for the convenience of the reader, we collect the important steps of the proof:

 \begin{thm}[Change of measure for continuous  local martingales]  \label{P change}
 	Let $\widetilde{\Omega} = C_1^\text{abs}([0,\infty),[0,\infty])$ be the set of  paths $\omega: [0, \infty) \rightarrow [0,\infty]$ with $\omega(0)=1$ that satisfy $\omega(t) = \omega(t \wedge \widetilde{T}_0(\omega) \wedge \widetilde{T}_\infty(\omega))$ for all $t \geq 0$, where $\widetilde{T}_0(\omega)$ and $\widetilde{T}_\infty(\omega)$ denote the first hitting times of $0$ and $\infty$ by $\omega$, and which are continuous on $[0, \widetilde{T}_\infty(\omega))$.
	Let $\{\widetilde{\CF}_t\}_{t \geq 0}$ denote the filtration generated by the canonical process $X$, defined by $X_t(\omega) := \omega(t)$ for all $t \geq 0$,  and set $\widetilde{\CF} = \bigvee_{t \geq 0} \widetilde{\CF}_t$. Let $\widetilde{\Prob}$ be a probability measure on $(\widetilde{\Omega}, \widetilde{\CF})$ such that $X$ is a (nonnegative) local $\widetilde{\Prob}$-martingale (starting in $1$). 
	
     Then there exists a unique probability measure $\Q$ on $(\widetilde{\Omega}, \widetilde{\CF})$ such that 
    \begin{align}   \label{E Qe}
        \E^{\Q}\left[\frac{1}{X_\rho} \left(Y\1_{\{1/X_\rho>0\}}\right)  \right] = \E^{\widetilde{\Prob}}\left[Y \1_{\{X_\rho>0\}}\right], 
    \end{align}
   where we set, for sake of notation, $\infty \cdot 0 := 0$,  for all random variables $Y$ taking values in  $[0, \infty]$ and being measurable with respect to $\widetilde{\CF}_\rho$ for some stopping time $\rho$ with $\rho \leq t$ for some $t\geq 0$.
     Furthermore,  if $X = \mathcal{E}(L)$ for some $\widetilde{\Prob}$-local martingale $L$ on $[0, \widetilde{T}_0)$ then $\widetilde{L} := L - \langle L\rangle$ is a $\Q$-local martingale  on $[0,\widetilde{T}_\infty)$ and $1/X = \mathcal{E}(-\widetilde{L})$.
\end{thm}

\begin {proof}
	Let $R_n$ denote the first hitting time of level $n$ by $X$, let $S_n$ denote the first hitting time of level $1/n$ by $X$, and set $\tau_n = R_n \wedge n$ for all $n \in \N$. Then, for all $n \in \N$, define a probability measure $\Q_n$ on $(\widetilde{\Omega}, \widetilde{\CF}_{\tau_n})$ by $\dd \Q_n = X_{\tau_n} \dd \Prob|_{\widetilde{\CF}_{\tau_n}}$ and observe that the family of probability measures $\{\Q_n\}_{n \in \N}$ is consistent, that is, $\Q_{n+i}|_{\widetilde{\CF}_{\tau_n}} = \Q_n$ for all $i,n \in \N$, and that  $\widetilde{\CF} = \widetilde{\CF}_{\widetilde{T}_0 \wedge \widetilde{T}_\infty} = \bigvee_{n \in \N} \widetilde{\CF}_{\tau_n}$ by using Lemma~1.3.3 in \citet{SV_multi}.  Exactly as in \citet{M} and Section~6 in \citet{F1972} it follows that $(\widetilde{\Omega}, \widetilde{\CF}_{\tau_n})$ is standard in the sense of Definition~V.2.2 of \citet{Pa} for all $n \in \N$ and that the remaining assumptions of the Extension Theorem~V.4.1 in \citet{Pa}
are satisfied	yielding the existence of a probability measure $\Q$
on $(\widetilde{\Omega}, \widetilde{\CF})$ such that $\Q|_{\widetilde{\CF}_{\tau_n}} = \Q_n$ for all $n \in \N$.

The fact that zero is an absorbing state of $X$ implies that
\begin{align*}
	\Q\left(X_\rho = 0\right) &= \lim_{n \uparrow \infty} \Q\left(\{X_\rho = 0\} \cap \{\rho < \tau_n\}\right) = \lim_{n \uparrow \infty} \Q_n\left(\{X_\rho = 0\} \cap \{\rho < \tau_n\}\right) \\
	&= \lim_{n \uparrow \infty}  \E^{\widetilde{\Prob}}\left[X_{\tau_n}  1_{\{X_\rho = 0\} \cap \{\rho < \tau_n\}}  \right]  = 0.
\end{align*}
This yields in conjunction with monotone convergence, for any stopping time $\rho$ with $\rho \leq t$ for some $t \geq 0$ and $A \in \widetilde{\CF}_\rho$,  that
    \begin{align*}  
        \E^{\Q}\left[\frac{1}{X_\rho} 1_A  \right] &= \E^{\Q}\left[\frac{1}{X_\rho} 1_{A \cap \{\rho < \widetilde{T}_\infty\}} \right]  = \lim_{n \uparrow \infty}  \E^{\Q}\left[\frac{1}{X_\rho} 1_{A \cap \{\rho < \tau_n\} \cap \{X_\rho > 0\}}  \right]  = 
 	 \lim_{n \uparrow \infty}  \E^{\widetilde{\Prob}}\left[\frac{X_{\tau_n}}{X_\rho} 1_{A \cap \{\rho < \tau_n\} \cap \{X_\rho > 0\}}  \right]  \\
	&= \lim_{n \uparrow \infty}   \widetilde{\Prob}\left(A \cap \{\rho < \tau_n\} \cap \{X_\rho > 0\} \right)
	 = \widetilde{\Prob}\left(A \cap\{X_\rho>0\}\right),
    \end{align*}
where the third equality follows from the observation that $\Q$ is absolutely continuous with respect to $\widetilde{\Prob}$ on $\widetilde{\CF}_{ \tau_n}$ with Radon-Nikodym derivative $X_{\tau_n}$ and the fourth equality follows by taking conditional expectation. 
Now, \eqref{E Qe} follows by another application of the monotone convergence theorem. The uniqueness of $\Q$ follows from plugging $\rho = \tau_n \wedge S_n$ and $Y = X_{\rho} \1_A$ for all $A \in \widetilde{\CF}_\rho$ and $n \in \N$
into  \eqref{E Qe}  and observing that $\widetilde{\CF} = \bigvee_{n \in \N} \widetilde{\CF}_{\tau_n \wedge S_n}$, similar to above.

It remains to show that $\widetilde{L}$, as defined in the statement, is a $\Q$-local martingale on $[0, \widetilde{T}_\infty)$. Towards this end, observe that $\Q(\lim_{n \uparrow \infty} \tau_n \wedge S_n =  \widetilde{T}_\infty)=1$.
Thus, it is sufficient to prove that $\widetilde{L}^{\tau_n \wedge S_n}$ is a $\Q$-local martingale. Using  that $\Prob$ and $\Q$ are equivalent on $\widetilde{\CF}_{\tau_n \wedge S_n}$, this follows directly from Girsanov's theorem; see for example Theorem~VIII.1.4 in \citet{RY}.
\end{proof}

Note that we may omit the indicators in \eqref{E Qe} if $X$ is a strictly positive true $\widetilde{\Prob}$-martingale by Girsanov's theorem. For general $\widetilde{\Prob}$-local martingales $X$, the event that $1/X$ hits zero might have positive $\Q$-probability; however, it has zero $\widetilde{\Prob}$-probability since $X$ is a $\widetilde{\Prob}$-local martingale.  The next corollary shall be essential:
\begin{cor}  \label{C iff}
	In the setup of Theorem~\ref{P change}, $X$ is a true $\widetilde{\Prob}$-martingale if and only if $\Q(1/X_t = 0) = 0$ for all $t > 0$.
\end{cor}
\begin{proof}
	The statement follows by plugging $\rho = t$ and $Y = X_t$ into \eqref{E Qe}.
\end{proof}
We remark that a similar statement as in Corollary~\ref{C iff} is already proven in Section~3.7 of \citet{McKean_1969}, under additional structural assumptions on the local martingale $L$.

\begin{remark}[Construction of canonical probability space]  \label{R canonical}
	A canonical probability space $(\widetilde{\Omega}, \widetilde{\CF}, \widetilde{\Prob})$ in the sense of Theorem~\ref{P change} can always be assumed when checking whether a continuous nonnegative local martingale $Z^L$ with $Z^L_0=1$, defined on some probability space $(\Omega, \CF, \Prob)$, is a (uniformly integrable) martingale. To see this, first note that $Z^L$ is a true martingale if and only if $\E^\Prob[Z^L_T] = 1$. Then define the mapping $\Theta: \Omega \rightarrow \widetilde{\Omega}$ by $\Theta(\omega) = Z^L(\omega)$, which is always well-defined, possibly after getting rid of a nullset. To complete
this transformation, define $\widetilde{\Prob} := \Prob \circ \Theta^{-1}$. Now, observe that the canonical process on $\widetilde{\Omega}$ has the same distribution under $\widetilde{\Prob}$ as $Z^L$ has under $\Prob$. In particular, 
the canonical process (defined on $\widetilde{\Omega}$) is a uniformly integrable martingale under $\widetilde{\Prob}$ if and only if $Z^L$ (defined on ${\Omega}$) is one under $\Prob$.  \qed
\end{remark}

\subsection{Proof of Theorem~\ref{C abstract}}
	We can assume, without loss of generality, first, that our probability space is the canonical one of Theorem~\ref{P change} by Remark~\ref{R canonical}, and second, that $T=1$, as
    we can always consider the local martingale $Z^L_{tT}$ for $T<\infty$ or $Z^L_{\tan(\pi t/2)}$ for $T = \infty$.  Thus, we need to prove that $Z^L$ is a true martingale on $[0,1]$.
    By Corollary~\ref{C iff} it is sufficient to show that $\Q(H) = 0$ for the probability measure
	$\Q$ of Theorem~\ref{P change} and for
    \begin{align*}
    	H := \left\{\mathcal{E}\left(-\widetilde{L}\right)_1 = 0\right\} = \left\{\left\langle \widetilde{L} \right\rangle_1 = \infty\right\} 
   \end{align*}
   $\Q$-almost surely, where the identity follows from Lemma~\ref{L extended} and where we have set $\widetilde{L} = L -\langle L \rangle$.
    Assume the opposite, to wit, $\Q(H) > 0$.
    Then observe that the sequence of stopping times $\{\tau_i\}_{i \in \N}$ defined as\footnote{See also the Addendum, where the definition of $\tau_i$ is modified to correct for an error in the proof.}
        \begin{align*} 
        \tau_i := \inf \left\{t \geq 0: f\left(\widetilde{L}_t + \left\langle \widetilde{L} \right\rangle_t,  \left\langle \widetilde{L} \right\rangle_t\right)  \exp\left(-\widetilde{L}_t - \frac{\left\langle \widetilde{L} \right\rangle_t}{2}\right)  \geq i\right \} \wedge \frac{i-1}{i}
    \end{align*}
    satisfies
    $\lim_{i \uparrow \infty} \widetilde{L}_{\tau_i} = \widetilde{L}_1$ on the complement of the set $H$ and,
        \begin{align} \label{E RV lim}
        \lim_{i \uparrow \infty}  f\left(\widetilde{L}_{\tau_i} + \left\langle \widetilde{L} \right\rangle_{\tau_i},  \left\langle \widetilde{L} \right\rangle_{\tau_i}\right)  \exp\left(-\widetilde{L}_{\tau_i}- \frac{\left\langle \widetilde{L} \right\rangle_{\tau_i}}{2}\right)   = \infty
    \end{align}
    on $H$. This holds because the continuous $\Q$-local martingale $\widetilde{L}$ on $[0, \widetilde{T}_0)$, where $\widetilde{T}_0$ denote the first hitting time of zero by $\mathcal{E}(-\widetilde{L})$,
 can be represented as a time-changed Brownian motion; to wit, $\widetilde{L}_t = B_{\langle\widetilde{L}\rangle_t}$ for $t < \widetilde{T}_0$ and some $\Q$-Brownian motion $B$.
    Thus, we obtain that
    \begin{align*}
        \infty &= \lim_{i \uparrow \infty} \E^\Q\left[f\left(\widetilde{L}_{\tau_i} + \left\langle \widetilde{L} \right\rangle_{\tau_i}, \left \langle \widetilde{L} \right\rangle_{\tau_i}\right)  \exp\left(-\widetilde{L}_{\tau_i}- \frac{\left\langle \widetilde{L} \right\rangle_{\tau_i}}{2}\right)\right] \\
        &= \lim_{i \uparrow \infty} \E^\Prob\left[f\left(\widetilde{L}_{\tau_i} + \langle {L} \rangle_{\tau_i},  \langle {L} \rangle_{\tau_i}\right) \1_{\{Z^L_{\tau_i} > 0\}} \right]
        \leq \sup_{\tau \in \mathcal{T}} \left\{\E^\Prob\left[f(L_\tau, \langle L \rangle_\tau)  \1_{\{Z^L_\tau > 0\}}\right]\right\} 
        < \infty  
    \end{align*}
    by Fatou's inequality, \eqref{E Qe}, and the assumption. The apparent contradiction gives $\Q(H)=0$.
\qed

We remark that the random variable on the left-hand side of \eqref{E RV lim} is finite $\Q$-almost surely if $X$ is a (uniformly integrable) martingale. However, this random variable nevertheless could have infinite expectation under $\Q$. This is exactly the situation when the condition of Theorem~\ref{C abstract} fails despite $X$ being a true martingale.

\section{A further analysis of the Novikov-Kazamaki conditions}  \label{S further}
In this section, we study the condition in \eqref{E cond}, which we reformulate in Subsection~\ref{SS modified}. Then, in Subsection~\ref{SS orders}, we introduce an ordering of local martingales according to the condition in \eqref{E cond}. Finally, in Subsection~\ref{SS deterministic}, we discuss \citet{Kazamaki_1977}'s
original condition, which only involves deterministic times.

\subsection{Modified Novikov-Kazamaki conditions}  \label{SS modified}
In this subsection, we shall derive a modified version of the condition in  \eqref{E cond}.
To begin with, we obtain the following useful result, similarly to Corollary~\ref{t novikov}. This observation generalizes Proposition~5 in \citet{Kazamaki_1983} to allow for certain functions $\phi$ with linear growth and for nonnegative local martingales $Z^{aL}$; see also Remark~\ref{R linear} in the appendix:
\begin{cor}[Martingale property of $Z^{aL}$] \label{C martingality}
	Fix any $a \in \R \setminus\{0,1\}$. Then the stochastic exponential $Z^{aL}$ is a (uniformly integrable) martingale on $[0,T]$ if \eqref{E cond} holds for some continuous function $\phi$ with $\liminf_{t \uparrow \infty}\phi^+(t)/t < |a-1|/2$ or, slightly more general, with $\phi(t) = |a-1|t/2  - \widetilde{\phi}(t)$, where $\widetilde{\phi}$ denotes a continuous function with $\limsup_{t \uparrow \infty} \widetilde{\phi}(t) = \infty$.
\end{cor}
\begin{proof}
	Similar to the proof of Corollary~\ref{t novikov}, define $f: \R \times [0,\infty) \rightarrow [0,\infty)$ by $$f(x,y) = \exp\left(x + \left(\frac{1}{2}-a\right)\frac{y}{a^2} - |a-1| \phi\left(\frac{y}{a^2}\right)\right)$$ for all $(x,y) \in \R \times [0,\infty)$ and let $B$ denote some Brownian motion. Then note that
$S^{L,a,\phi} = f(aL, a^2\langle L \rangle) \1_{\{Z^{aL} > 0\}}$ and that
        \begin{align*} 
            \limsup_{t \uparrow \infty} f(B_t+t,t) \cdot \exp\left(-B_t - \frac{t}{2}\right) = \limsup_{t \uparrow \infty} \exp\left(|a-1| \left(\frac{|a-1|}{2} \cdot \frac{t}{a^2} - \phi\left(\frac{t}{a^2}\right) \right) \right) = \infty.
        \end{align*}
	Thus, an application of Theorem~\ref{C abstract} yields the statement.
\end{proof}

We directly obtain the following modified version of the Novikov-Kazamaki conditions:
\begin{cor}[Modified Kazamaki's criterion]  \label{c novikov2}
	The stochastic exponential $Z^L$ is a (uniformly integrable) martingale on $[0,T]$ if, for some $a \in \R \setminus\{0,1\}$ and some continuous function $\phi$ as in Corollary~\ref{C martingality}, we have that
    \begin{align} \label{E cond mod}
        \sup_{\tau \in \mathcal{T}} \left\{\E\left[S^{L/a,a,\phi}_{\tau} \right]\right\}  < \infty
    \end{align}
    with $\mathcal{T}$ as in Corollary~\ref{t novikov}.
\end{cor}
\begin{proof}
    The statement follows from Corollary~\ref{C martingality} after replacing $L$ by $L/a$.
\end{proof}
For example, using $a = 1/2$ and $\phi(x) = d x$ for some $d \in [0, 1/4)$, \eqref{E cond mod} simplifies to
    \begin{align}  \label{E cond mod2}
    \sup_{\tau \in \mathcal{T}} \left\{\E\left[\exp(L_\tau - 2 d \langle L \rangle_\tau) \1_{\{Z^L_\tau>0\}} \right]\right\} < \infty.
    \end{align}
 For illustration, consider the case of $L$ being a Brownian motion, stopped as soon as it hits $c + 2 d t$ for some constant $c>0$. Obviously, the condition in \eqref{E cond mod2} then holds; therefore, we have that $Z^L$ is a uniformly integrable martingale, a well-known result; see \citet{Shepp_1969}, \citet{Shiryaev_Vostrikova}, \citet{ChernyShiryaev}, and Examples~1 and 2 in \citet{Kazamaki_1983} for more general statements in this context. Another application of the last corollary directly yields the following observation:
\begin{cor}[Another criterion]  
	The stochastic exponential $Z^L$ is a (uniformly integrable) martingale on $[0,T]$ if for some strictly positive continuous function $\psi: [0, \infty] \rightarrow (0, \infty]$ with $\limsup_{t \uparrow \infty} \psi(t) = \infty$,  we have that
    \begin{align*} 
        \sup_{\tau \in \mathcal{T}} \left\{\E\left[Z^L_{\tau} \psi(\langle L\rangle_\tau) \right]\right\}  < \infty
    \end{align*}
    with $0 \cdot \infty := 0$.
\end{cor}
\begin{proof}
    The statement follows from Corollary~\ref{c novikov2} with $a=2$ and $\phi(t) = t/2  - \log(\psi(4t))$.
\end{proof}

Corollary~\ref{C martingality} also yields the following equivalent formulation of the condition in \eqref{E cond}:
\begin{cor}[Submartingality of $S^{L,a,0}$]  \label{C subm}
    The condition in \eqref{E cond} holds for $\phi \equiv 0$ if and only if $S^{L,a,0}$ is a $\Prob$-submartingale on $[0,T]$.
\end{cor}
\begin{proof}
    First, observe that the submartingality of $S^{L,a,0}$ on $[0,T]$ implies \eqref{E cond} directly. For the reverse direction, note  that $Z^{aL}$ is a true $\Prob$-martingale by Corollary~\ref{C martingality} and generates a new probability measure $\Q^a$ via Girsanov's theorem. This observation and the fact that
        \begin{align}  \label{E alt representation}
        S^{L,a,0}
          &= Z^{aL}\cdot \exp\left( \frac{1}{2}\left(a - 1\right)^2 \langle L\rangle\right) \1_{\{Z^{aL}>0\}}
    \end{align}
 yield that
     \begin{align*}
        \E^\Prob\left[S^{L,a,0}_T\right] = \lim_{t \uparrow T} \E^{\Q^a}\left[\exp\left( \frac{1}{2}\left(a - 1\right)^2 \langle L\rangle_t\right)\right] \leq \sup_{\tau \in \mathcal{T}} \left\{\E^\Prob\left[S^{L,a,\phi}_{\tau} \right]\right\}  < \infty,
    \end{align*}
    and similarly that $S^{L,a,0}$ is a $\Prob$-submartingale.
\end{proof}

\subsection{Novikov-Kazamaki orders}   \label{SS orders}
In the following, we classify the local martingales $L$ that satisfy the condition in \eqref{E cond}:
\begin{defin}[Local martingales of (Novikov-Kazamaki) order $a$]
We call a local martingale $L$ a local martingale of (Novikov-Kazamaki) order $a$ for some $a \in \R$ with respect to some measurable function $\phi$ if \eqref{E cond}  is satisfied for this choice of $a$ and $\phi$.
We denote by $\mathcal{NK}^\phi(a)$ the class of all local martingales of order $a$ with respect to $\phi$. \qed
\end{defin}

It is clear that $\mathcal{NK}^\phi(a)$ contains all constant local martingales $L \equiv \text{const}$; thus, $\mathcal{NK}^\phi(a) \neq \emptyset$. Furthermore, if $\phi$ is bounded from below, we have that $L \in \mathcal{NK}^\phi(1)$ for any local martingale $L$.
 Since Novikov's condition implies Kazamaki's criterion, we further have $\mathcal{NK}^0(0) \subset \mathcal{NK}^0(1/2) \subset \mathcal{NK}^0(1)$. The
next corollary generalizes this observation:
\begin{cor}[Novikov-Kazamaki orders]\label{C orders}
    For $a < b < 1 < c < d$ and for any continuous lower function $\phi$ we have
    \begin{align*}
        \mathcal{NK}^\phi(a) \subset \mathcal{NK}^\phi(b) \text{ and }
        \mathcal{NK}^\phi(c) \supset \mathcal{NK}^\phi(d),
    \end{align*}
    where all inclusions are strict if $\phi \equiv 0$.
\end{cor}
\begin{proof}
	Fix $e \in \R \setminus \{1\}$ and $f \in (e,1)$ or $f \in (1,e)$ depending on the sign of $e-1$ and $L \in \mathcal{NK}^\phi(e)$. Then $Z^L$ is a (uniformly integrable) martingale by Corollary~\ref{t novikov} and defines a new probability measure $\Q$ by $\dd \Q = Z^L_T \dd \Prob$. An application of Jensen's inequality yields
    \begin{align*}
        \E^\Prob\left[S^{L,f,\phi}_{\tau}\right]
        &= \E^\Q\left[\exp\left((f-1) \widetilde{L}_\tau - |f-1| \phi(\langle L \rangle_\tau) \right)   \right]
         = \E^\Q\left[\exp\left((e-1) \widetilde{L}_\tau - |e-1| \phi(\langle L \rangle_\tau) \right)^{\frac{f-1}{e-1}}   \right]\\
       &\leq \left(\E^\Prob\left[S^{L,e,\phi}_{\tau}\right]\right)^{\frac{f-1}{e-1}}
    \end{align*}
	with $\widetilde{L} = L - \langle L \rangle$, for all $\tau \in \mathcal{T}$. This yields the asserted inclusions. 
    The strictness of the inclusions follows from Example~\ref{Ex strictness}.
\end{proof}
The last result is, apart from the claim of the strictness of the inclusions, Proposition~1 in \citet{Kazamaki_1983}.

\subsection{Kazamaki's criterion with deterministic times}   \label{SS deterministic}
In his original paper, \citet{Kazamaki_1977}  only considered true martingales $L$. He then formulated his criterion without stopping times; more precisely, he showed that it is sufficient for the martingale property of $Z^L$ to only require
\eqref{E cond} for constant times $\tau \equiv c \in [0,T)$ if $L$ is a true martingale. Indeed, if $L$ is a martingale and $a = 1/2$, \eqref{E cond} then follows directly from Jensen's inequality. However, for $L$ only a local martingale the finite supremum over deterministic times is usually not sufficient. Example~\ref{Ex 2dBessel} below illustrates this point. The precise result is as follows:
\begin{prop}[Deterministic times]  \label{P weak}
    If either
    \begin{align}  \label{E prop cond}
        \sup_{t \in [0,T)} \E^\Prob\left[S^{L,a,0}_t\right]<\infty \text{ or }
        \E^\Prob\left[S^{L,a,0}_T\right]<\infty,
    \end{align}
    then the following conditions are equivalent for any $a \in \R$:
    \begin{enumerate}
        \item[(i)] the process $Z^{aL}$ is a uniformly integrable $\Prob$-martingale;
        \item[(ii)] the process $S^{L,a,0}$ is a $\Prob$-submartingale on $[0,T]$.
    \end{enumerate}
    Furthermore, any of these conditions then implies that
    \begin{enumerate}
        \item[(iii)] the process  $Z^L$ is a uniformly integrable $\Prob$-martingale.
    \end{enumerate}
\end{prop}
\begin{proof}
    The fact that (i) implies (ii) follows from \eqref{E alt representation} and the computation
    \begin{align*}
         \E^\Prob\left[\left.S^{L,a,0}_t\right|\CF_s\right] = Z^{aL}_s \E^{\Q^a}\left[\left.\exp\left( \frac{1}{2}\left(a - 1\right)^2 \langle L\rangle_t\right)\right|\CF_s\right]
        \geq Z^{aL}_s  \exp\left( \frac{1}{2}\left(a - 1\right)^2 \langle L\rangle_s\right) = S^{L,a,0}_s
    \end{align*}
    for $s < t$, where $\Q^a$ is defined by $\dd \Q^a =Z^{aL}_T \dd \Prob$. The finiteness of $\E^\Prob[S^{L,a,0}_T]$ follows as in Corollary~\ref{C subm}.
    Corollaries~\ref{C martingality} and \ref{C subm} yield the reverse direction.
    The necessity of (iii) is basically the statement of
    Corollary~\ref{t novikov}.
\end{proof}
Example~\ref{Ex 2dBessel II} illustrates that (iii) does not necessarily imply (ii) or (i).
Indeed, in order to use the same argument as in the step from (i) to (iii), one would need, given the martingale property of $Z^L$ with corresponding measure $\Q$, a condition like
\begin{align*}
    \E^\Q\left[\exp\left(\frac{1}{2} (a-1)^2 \left\langle {L} \right\rangle_T\right)\right] = \E^\Prob\left[S^{aL,1/a,0}_T\right]<\infty,
\end{align*}
replacing \eqref{E prop cond}, which translates into
\begin{align*}
    \E^\Q\left[\exp\left((a-1)\widetilde{L}_T\right)\right] = \E^\Prob\left[S^{L,a,0}_T\right]<\infty,
\end{align*}
where $\widetilde{L} = L - \langle L \rangle$.

 If $\sup_{t \in [0,T)} \E^\Prob[S^{L,a,0}_t]=\infty$, no conclusions can be drawn. Indeed, in Remark~\ref{R quad} below, we discuss two processes $L^{(1)}, L^{(2)}$, for which this supremum is infinite for $a=0$ but one of them generates a martingale through stochastic exponentiation, the other one does not.

%

\section{Examples}
In this section, we discuss several examples to highlight some of the results of the first sections. To begin with, as we allow for local martingales $L$ such that $Z^L$ has positive probability to hit zero in the criterion in \eqref{E cond}, we now provide an example of a nonnegative martingale $Z^L$ hitting zero to which the criterion could be applied to:
\begin{ex}[Stopped Brownian motion]
	The goal of this example is to show that the condition in \eqref{E cond} can be applied to the case of a Brownian motion stopped when it hits zero.  Towards this end, set
	\begin{align*}
		L_\cdot =  \1_{\{B_\cdot  > 0\}} \int_0^\cdot \frac{1}{B_t} \dd B_t = \1_{\{B > 0\}}  \left(\log(B_\cdot ) + \frac{1}{2} \int_0^\cdot \frac{1}{B_t^2} \dd t\right),
	\end{align*}
	where $B$ denotes a Brownian motion stopped when it hits zero. We know from optional stopping that $Z^L = B$ is a martingale over any finite time horizon. However, we shall not use this prior knowledge but instead check the criterion in \eqref{E cond} directly.
	Let us consider the case $a = 2$. We obtain that
	\begin{align*}
		S^{L,2,0}_\cdot = \1_{\{B_\cdot > 0\}}  B^2_\cdot  \exp\left(-\frac{1}{2}\int_0^\cdot \frac{1}{B_t^2} \dd t\right) \leq B^2_\cdot.
	\end{align*}
	Thus, \eqref{E cond}  now holds since $B^2$ is a submartingale over any finite time horizon $T<\infty$.
%
\qed
\end{ex}

From now on, we shall always assume that $T=1$.
The next example illustrates how Proposition~\ref{P weak} can be applied to check the martingale property of a stochastic exponential:
\begin{ex}[Iterative application of Kazamaki's criterion]  \label{Ex iterative}
    In this example, we study a family $\{L^{(c)}\}_{c \in \R}$ of $\Prob$-martingales
    and their corresponding stochastic exponentials. To begin with, we introduce the $\Prob$-martingale $I$ by
    \begin{align*}
        I_t := \int_0^{t} {B}_s \dd {B}_s = \frac{1}{2} ({B}^2_{t} - t),
    \end{align*}
	where $B$ denotes a Brownian motion, 
    and observe that
    Formula~1.9.3(1) on page~168 in \citet{Borodin_handbook} yields that
     \begin{align}  \label{E I int}
        \E^{\Prob}\left[\exp(\alpha I_t - \beta \langle I \rangle_t)\right] < \infty
    \end{align}
    for all $\alpha \in \R$ and $\beta \in (0, \infty)$. Now, set $L^{(c)} := c I$ for some $c \in \R$
    and observe that Kazamaki's criterion in \eqref{E Kaz criterion} holds, due to the martingality of $L^{(c)}$, if and only if $c < 2$; as otherwise
    \begin{align*}
    	\E^{\Prob}\left[\exp\left(\frac{L^{(c)}_1}{2}\right)\right] = \E^{\Prob}\left[\exp\left(\frac{c B_1^2}{4}\right)\right] \exp\left(-\frac{c}{4}\right) = \infty.
    \end{align*}

    Now consider the
    case $c \geq 2$. We want to prove that $Z^{L^{(c)}}$ is a $\Prob$-martingale.  Towards this end, set $a = 3/4$ in \eqref{E cond} and check that $\E[S^{L^{(c)}, 3/4,0}_1] < \infty$ by \eqref{E I int}. Thus, Proposition~\ref{P weak} yields that is sufficient to
    check whether $Z^{L^{(c_1)}}$ is a martingale for $c_1 = 3c/4$. If $c_1 < 2$, we are done as above. Otherwise, we iterate the argument until eventually $c_n = (3/4)^n c < 2$ for some sufficiently large $n \in \N$.

    As $B$ under the measure generated by $Z^{L^{(c)}}$ has Ornstein-Uhlenbeck dynamics, this example shows that the Wiener and Ornstein-Uhlenbeck measures are equivalent on finite time horizons.
    We also refer to Exercise~IX.2.10 in \citet{RY} for a different argument based on a study of the explosion time of a certain diffusion. \qed
\end{ex}

We now construct local martingales with different Novikov-Kazamaki orders as introduced in Subsection~\ref{SS orders}:
\begin{ex}[Novikov-Kazamaki orders]  \label{Ex strictness}
	We want to construct a family of martingales $\{\widehat{L}^{(a)}\}_{a \in \R\setminus\{1\}}$ such that $\widehat{L}^{(a)} \notin  \mathcal{NK}^0(a)$
	but $\widehat{L}^{(a)} \in  \mathcal{NK}^0(b)$ for all $b \in (a,1]$ or $b \in [1,a)$, depending on the sign of $a-1$.
	Towards this end, we modify Example~\ref{Ex iterative}. To begin with, we introduce a family $\{\widetilde{B}^{(a)}\}_{a \in \R\setminus\{1\}}$ of $\Prob$-Ornstein-Uhlenbeck processes with
$\widetilde{B}^{(a)}_0 = 0$ and with dynamics
\begin{align*}
    \dd \widetilde{B}^{(a)}_t = - \frac{1}{a-1} \widetilde{B}^{(a)}_t \dd t + \dd B_t,
\end{align*}
where $B$ denotes again a $\Prob$-Brownian motion. We now
consider the family of $\Prob$-local martingales $\{\widehat{L}^{(a)}\}_{a \in \R\setminus\{1\}}$ defined as
	\begin{align*}
		\widehat{L}^{(a)}_t := \frac{1}{a-1} \int_0^t \widetilde{B}^{(a)}_s \dd B_s.
	\end{align*}
The equivalence of the Wiener and Ornstein-Uhlenbeck measure, which we observed in Example~\ref{Ex iterative}, yields that $Z^{\widehat{L}^{(a)}}$ is a $\Prob$-martingale and thus generates a probability measure 
$\Q^{(a)}$ by $\dd \Q^{(a)} = Z^{\widehat{L}^{(a)}}_1 \dd \Prob$. Define the $\Q^{(a)}$-martingale
\begin{align*}
    \widetilde{L}^{(a)}_t := \widehat{L}^{(a)}_t -
\left\langle \widehat{L}^{(a)}\right\rangle_t = \frac{1}{a-1} \int_0^t \widetilde{B}^{(a)}_s \dd \widetilde{B}_s^{(a)} = \frac{1}{2(a-1)} \left(\left({\widetilde{B}^{(a)}_t }\right)^2 - t\right),
\end{align*}
where $\widetilde{B}^{(a)}$ is a $\Q^{(a)}$-Brownian motion, and observe that
\begin{align*}
    \E^\Prob\left[S^{\widehat{L}^{(a)},b,0}_1\right] &= \E^{\Q^{(a)}}\left[\exp\left((b-1)\widetilde{L}^{(a)}_1\right)\right] = \E^{\Q^{(a)}}\left[\exp\left(\frac{b-1}{a-1} \cdot \frac{\left(\widetilde{B}_1^{(a)}\right)^2}{2}\right)\right]
	\cdot \exp\left(-\frac{b-1}{a-1} \cdot \frac{t}{2}\right),
\end{align*}
which is finite for all $b \in (a,1]$ or $b \in [1,a)$, but infinite for $b=a$. Thus,
$\widehat{L}^{(a)} \notin  \mathcal{NK}^0(a)$, but Corollary~\ref{C subm} implies that  $\widehat{L}^{(a)} \in  \mathcal{NK}^0(b)$ for all $b \in (a,1]$ or $b \in [1,a)$, respectively.
\qed
\end{ex}

The next example discusses a local martingale $\widehat{L}$ for which $\exp(\widehat{L})$ is not a submartingale, despite having finite expectation:
\begin{ex}[Two-dimensional Bessel process I]\label{Ex 2dBessel}
In order to be consistent with Example~\ref{Ex 2dBessel II} below, we here work under a probability measure $\Q$. Let $R$ denote a two-dimensional $\Q$-Bessel process starting in $R_0=1$ with dynamics
\begin{align*}
	\dd R_t = \frac{1}{2 R_t} \dd t + \dd \widetilde{B}_t,
\end{align*}
where $\widetilde{B}$ denotes a $\Q$-Brownian motion.  Existence and uniqueness of the solution to this SDE is guaranteed by the results in Section~3.3.C of \citet{KS1}.
Let us study the local martingale
\begin{align}  \label{E 2d tildeL}
    \widetilde{L}_\cdot := -\frac{1}{2} \int_0^\cdot \frac{1}{R_t} \dd \widetilde{B}_t = - \frac{\log(R_\cdot )}{2}.
\end{align}
The $\Q$-local martingale $Z^L$ cannot be a true $\Q$-martingale. If it were, $R$ would be a Brownian motion under the corresponding measure and thus hit zero with positive probability. This event, however, has probability zero under $\Q$. Therefore, $Z^L$ is a strict $\Q$-local martingale.

 Proposition~\ref{P weak} now yields that $Y := S^{\widetilde{L},1/2,0} = \exp(\widetilde{L}/2) = R^{-1/4}$ is not a $\Q$-submartingale, even given that we can check that
\begin{align} \label{E 2d ineq}
     \E^{\Q}\left[Y_t\right] = \E^{\Q}\left[R^{-1/4}_t\right] < \infty
\end{align}
for all $t \in [0,1]$;
see also Exercise~3.3.37 in \citet{KS1}. We emphasize that $-\log(R)/4$ is a local $\Q$-martingale and thus, $Y$ has
a strictly positive drift:
\begin{align*}
    \dd Y_t = \frac{1}{32} Y_t^{9} \dd t  -\frac{1}{4} Y_t^5 \dd \widetilde{B}_t .
\end{align*}
Thus, $Y$ is not a $\Q$-supermartingale either. 
Otherwise, it also would be a local $\Q$-supermartingale. This is, however, not possible due
to its strictly positive drift.
%
\qed
\end{ex}
The next example continues the discussion in Example~\ref{Ex 2dBessel} in order to provide an example for the lack of sufficiency of (iii) for (ii) in Proposition~\ref{P weak}:
\begin{ex}[Two-dimensional Bessel process II]\label{Ex 2dBessel II}
    We continue our discussion of the two-dimensional $\Q$-Bessel process of Example~\ref{Ex 2dBessel}. For $\widetilde{L}$ defined in \eqref{E 2d tildeL} compute
    \begin{align*}
        \dd Z^{-\widetilde{L}}_t = \frac{1}{2 \sqrt{R_t}} \exp\left(-\frac{1}{8} \int_0^t \frac{1}{R_s^2} \dd s\right) \dd \widetilde{B}_t
    \end{align*}
    and note that
    \begin{align*}
        \E^\Q\left[\sqrt{\left\langle Z^{-\widetilde{L}}\right\rangle_1}\right] &\leq \E^\Q\left[ \sqrt{ \int_0^1  \frac{1}{2 R_t}\dd t }\right] \leq \sqrt{ \E^\Q\left[ R_1\right]} < \infty.
    \end{align*}
	Now, the Burkholder-Davis-Gundy inequalities \citep[see for example Theorem~3.3.28 in][]{KS1} or the results in \citet{Elworthy_Li_Yor_97} (see \eqref{E ELY} below) imply 
	 that $Z^{-\widetilde{L}}$ is a $\Q$-martingale and thus defines an equivalent  probability measure $\Prob$ by $\dd \Prob / \dd \Q =
    Z^{-\widetilde{L}}_1$. 

We then have that $Z^L$ is a $\Prob$-martingale, where we set
    $L = \widetilde{L} + \langle L \rangle$. Let us now consider $a = 3/2$ in the Novikov-Kazamaki criterion. We obtain
    $\E^\Prob[S^{L,3/2,0}_1] = \E^\Q[\exp(\widetilde{L}_1/2)] < \infty$, where  the inequality is the same as in \eqref{E 2d ineq}.  However, $S^{L,3/2,0}$ is not a $\Prob$-submartingale since
    $\exp(\widetilde{L}/2)$ is not a $\Q$-submartingale, as discussed in Example~\ref{Ex 2dBessel}. This illustrates that (iii) does not necessarily imply (ii) in
    Proposition~\ref{P weak}.
\qed
\end{ex}

\begin{remark}[On the quadratic variation]  \label{R quad}
The representation of the expectation of a nonnegative continuous local martingale $Z$ as
\begin{align}  \label{E ELY}
    \E[Z_T] = Z_0 - \lim_{y \uparrow \infty} \left(y \Prob\left(\sqrt{\langle Z\rangle_T} \geq y\right)\right)
\end{align}
for any $T>0$ in \citet{Elworthy_Li_Yor_97} implies that two continuous local martingales $Z^{L^{(1)}}$ and $Z^{L^{(2)}}$ with identically distributed quadratic variations are either both true martingales or both strict local martingales.
    This observation and Novikov's condition, which is a condition on the quadratic variation of the stochastic logarithm $L$ of a strictly positive continuous local martingale $Z^L$,
     raise the question whether the true martingality of $Z^{L^{(1)}}$  also implies the one of $Z^{L^{(2)}}$
    if  the quadratic variation processes of the logarithms agree, that is, if
    $\langle L^{(1)}\rangle \equiv \langle L^{(2)}\rangle$.

    A simple counter-example, complementing the one in \citet{Kazamaki_1977}, is provided in Examples~\ref{Ex 2dBessel} and \ref{Ex 2dBessel II}, where a $\Q$-local martingale
    $\widetilde{L}$ leads to a true $\Q$-martingale $Z^{-\widetilde{L}}$, but a strict $\Q$-local martingale $Z^{\widetilde{L}}$ despite the
    obvious fact that  $\langle \widetilde{L}\rangle \equiv \langle -\widetilde{L}\rangle$. One might now wonder whether the loss of the martingale property is due to the change of sign in the logarithm, which also changes the instantaneous correlation of the stochastic exponential with the driving Brownian motion. Although this is, by construction of the example, true in this specific case,  there exist examples of local martingales $L^{(1)}, L^{(2)}$ with $\text{sign}(L^{(1)}) = \text{sign}(L^{(2)})$ and  $\langle L^{(1)}\rangle \leq \langle L^{(2)}\rangle$ such that  $Z^{L^{(1)}}$ is a strict local martingale while  $Z^{L^{(2)}}$ is a true martingale; or  such that  $Z^{L^{(1)}}$ is a non-uniformly integrable local martingale  while  $Z^{L^{(2)}}$ is a uniformly integrable martingale.

    One such example is discussed on page~297 in \citet{Kazamaki_1983}. There, a Brownian stopping time $\tau$ is constructed such that $L = B^\tau$ for some Brownian motion $B$ leads to a non-uniformly integrable local martingale $Z^L$. On the other hand, by means of Corollary~\ref{C martingality}, it can be shown that $Z^{2L}$ is a uniformly integrable martingale.  Another such example  is constructed in \citet{DS_counter}, where two local martingales $L^{(1)}, L^{(2)}$ are considered with $\langle L^{(1)} + L^{(2)}\rangle = \langle L^{(1)}\rangle  + \langle L^{(2)}\rangle \geq \langle L^{(1)}\rangle$ such that $Z^{L^{(1)}}$ is a non-uniformly integrable local martingale, but $Z^{L^{(1)} + L^{(2)}}$ is a uniformly integrable martingale.  
\qed
\end{remark}

In the next example, we study the martingale property of stochastic exponentials related to the three-dimensional Bessel process:
\begin{ex}[Three-dimensionsal Bessel process]
We study the three-dimensional Bessel process, denoted here by $R$, with initial value $R_0=1$ and with dynamics
\begin{align*}
	\dd R_t = \frac{1}{R_t} \dd t + \dd B_t
\end{align*}
for some Brownian motion $B$.  Existence and uniqueness of the solution to this SDE is again guaranteed by the results in Section~3.3.C of \citet{KS1}.

Let us consider the local martingales
\begin{align}  \label{E Def L}
	L^{(1)}_\cdot := \int_0^\cdot \frac{1}{R_t} \dd B_t = \log(R_\cdot ) - \frac{1}{2} \int_0^\cdot \frac{1}{R^2_t} \dd t, \text{ } L^{(2)} = -L^{(1)}
\end{align}
and the corresponding stochastic exponentials $Z^{L^{(1)}}$ and $Z^{L^{(2)}}$, where the identity in \eqref{E Def L}
 follows from It\^o's rule.
 Since $\sup_{0 \leq t < \infty} \E^\Prob[1/R^2_t] < \infty$,
the local martingales $L^{(1)}$ and $L^{(2)}$
 are actually true martingales on any finite time horizon; see Exercise~II.20(d) in \citet{Protter} and Section~3.2 of \citet{KS1}.
 It is clear that $\langle L^{(1)}\rangle \equiv \langle L^{(2)}\rangle$. Let us now compute
 $Z^{L^{(1)}}$ and $Z^{L^{(2)}}$:
 \begin{align*}
 	Z^{L^{(1)}}_t
	&= \exp\left( \log(R_t) - \frac{1}{2} \int_0^t \frac{1}{R^2_s} \dd s-  \frac{1}{2} \int_0^t \frac{1}{R^2_s}\dd s\right)
	= R_t \exp\left( -\int_0^t \frac{1}{R^2_s} \dd s\right)
	= 1+\int_0^t \exp\left( -\int_0^s \frac{1}{R^2_u} \dd u\right) \dd B_s,  \\
 	Z^{L^{(2)}}_t
	&= \exp\left(- \log(R_t) + \frac{1}{2} \int_0^t \frac{1}{R^2_s} \dd s-  \frac{1}{2} \int_0^t \frac{1}{R^2_s} \dd s\right)
	= \frac{1}{R_t}.
\end{align*}
It is well-known that the reciprocal of a three-dimensional Bessel process is a strict local martingale; see Exercise~3.3.36 in \citet{KS1}. However, since $Z^{L^{(1)}}$ can be represented as a stochastic integral with respect to Brownian motion of a bounded, continuous process, it is a true martingale. This yields another example for two true martingales $L^{(1)}, L^{(2)}$, such that $\langle L^{(1)}\rangle \equiv \langle L^{(2)}\rangle$, but $Z^{L^{(1)}}$ is a true martingale while $Z^{L^{(2)}}$ is not.

Indeed, the quadratic variations
\begin{align*}
	\langle Z^{L^{(1)}} \rangle_t  &= \int_0^t \exp\left( -2 \int_0^s \left(Z_u^{L^{(2)}}\right)^2 \dd u\right) \dd s, \\
	\langle Z^{L^{(2)}} \rangle_t &= \int_0^t  \left(Z_s^{L^{(2)}}\right)^4 \dd s
\end{align*}
have quite different tail behavior.  We remark that $Z^{L^{(1)}}$ is one of these instances for which Novikov's condition does not hold (since otherwise $Z^{L^{(2)}}$ would be a true martingale), but $Z^{L^{(1)}}$ is a true martingale.

Let us now study Kazamaki's criterion, which states that $\mathcal{E}(L)$ is a true
martingale for some local martingale $L$ if $\exp(L/2)$ is a submartingale; see Corollary~\ref{C subm}. This condition is also sufficient, although not necessary, for $\mathcal{E}(L)$ being a true martingale, and is weaker than
Novikov's condition; see Corollary~\ref{C orders}.
To start, consider the two processes $C = \exp(L^{(1)}/2)$ and $D = \exp(L^{(2)}/2) = 1/C$. It\^o's
formula yields the
dynamics
\begin{align*}
    \dd C_t &= \frac{C_t}{2 R_t} \left(\dd B_t + \frac{1}{4 R_t} \dd t\right),\\
    \dd D_t &= \frac{D_t}{2 R_t} \left(-\dd B_t + \frac{1}{4 R_t} \dd t\right);
\end{align*}
these dynamics look very similar.

We argued above that both $L^{(1)}/2$ and $L^{(2)}/2$ are true martingales. Exponentials
of martingales are, by Jensen's inequality, submartingales, provided they are integrable.
This observation, and the fact that $Z^{L^{(2)}}$ is not a martingale, yields directly that
$\E^\Prob[D_t] = \infty$ for all $t>0$.  On the other side, we obtain from \eqref{E Def L} that
$0 \leq C_t \leq \sqrt{R_t}$. However, $R$ has finite positive moments; indeed the moments of
$R^p$ and $\widetilde{B}^{p + 1}$ agree for any $p>-1$, where $\widetilde{B}$ denotes a Brownian motion starting in $R_0$ and being stopped in zero. This follows from the well-known connection of Brownian and Bessel measure;
see for example \citet{Perkowski_Ruf}.
We remark that the moments of $C$ can also be explicitly computed by means of Formula~1.20.8 on page~386 in
\citet{Borodin_handbook}.  Thus, $Z^{L^{(1)}}$ represents an example that does not satisfy
Novikov's condition, but satisfies Kazamaki's criterion.

One might wonder what the dynamics of $R$ are under the probability measure $\widetilde{\Q}$ corresponding
to the Radon-Nikodym derivative $Z^{L^{(1)}}$. By Girsanov's theorem, $R$ has dynamics
\begin{align*}
	\dd R_t = \frac{2}{R_t} \dd t + \dd \widetilde{B}_t
\end{align*}
where $\widetilde{B}$ is a $\widetilde{\Q}$-Brownian motion. Thus, $R$ is a $\widetilde{\Q}$-Bessel process of dimension five; see Section~3.3.C of \citet{KS1}.

In Subsection~\ref{SS deterministic}, we discussed the obvious fact that the supremum over deterministic times is, in general, smaller than the supremum over stopping times. The three-dimensional Bessel process yields a simple illustration of this fact. We already observed that $\sup_{0 \leq t \leq 1} \E^\Prob[1/R^2_t] < \infty$.  Consider now any $n \in \N$ and the first hitting time $\tau_n$ of $1/n$ by $R$ and infinity otherwise. Then we obtain that
$\E^\Prob[1/R^2_{\tau_n \wedge 1}] \geq n \E^{\Prob}[1/R_{\tau_n \wedge 1} \1_{\{\tau_n \leq 1\}}] = n \widetilde{\Prob}(\tau_n \leq 1) \rightarrow \infty$ as $n \uparrow \infty$, where $\widetilde{\Prob}$ denotes the probability measure under which $R$ is Brownian motion.  Here, $\widetilde{\Prob}(\tau_n \leq 1)$ does not tend to zero as it represents the probability of a Brownian motion started in $1$ to hit $0$ before time $1$. 
\qed
\end{ex}%
\appendix
\section{Lower functions}  \label{A lower}
The formulation of Corollary~\ref{t novikov} contains the notion of lower functions, which we
briefly recall here.

\begin{defin}[Lower and upper function] \label{D lower}
    Let $B$ denote a Brownian motion on some probability space $(\Omega, \CF, \Prob)$ and let $\phi: [0,\infty) \rightarrow \R$ be a continuous function. Define the event
    \begin{align*}
        G := \{\omega \in \Omega: B_s(\omega) < \phi(s) \text{ for all } s \geq t(\omega) \text{ for some } t(\omega)>0\}.
    \end{align*}
    If $\Prob(G ) = 0$ ($\Prob(G) = 1$), then $\phi$ is called a lower (upper) function.
\end{defin}
Due to Blumenthal's zero-one law we have that either $\Prob(G) = 0$ or $\Prob(G) = 1$, thus any continuous function is either a lower or an upper function; see Section~1.8 in \citet{ItoMcKean} and Section~2 in \citet{ChernyShiryaev}. Lower functions are, for example,
all constant functions or the functions $\phi(t) = C \sqrt{t}$ or $\phi(t) = \sqrt{2 t \log(\log(t))}$; this can be checked by an application of Kolmogorov's test; see Problem~1.8.3 in \citet{ItoMcKean}.

The following result appears as Lemma~2.2 in \citet{ChernyShiryaev}. It is a corollary of Girsanov's formula.
\begin{lemma}[Limits involving lower and upper functions]  \label{L lower}
    If $\phi_1$ is a continuous lower function and $\phi_2$ is a continuous upper function,
    then
    \begin{align*}
        \limsup_{t \uparrow \infty} (B_t - \phi_1(t)) = \infty; \text{ and }
         \limsup_{t \uparrow \infty} (B_t - \phi_2(t)) = -\infty.
    \end{align*}
\end{lemma}
\begin{remark}[Functions of linear growth] \label{R linear}
	Observe
 that $\liminf_{t \uparrow \infty}\phi^+(t)/t = 0$ for any lower function $\phi$, where we denote by $\phi^+$ the positive part of a function $\phi$. However,  the function $\phi(t) = \sqrt{3 t \log(\log(t))}$ illustrates that sublinear growth is not sufficient for a function being lower.  \qed
\end{remark}

\bibliographystyle{apalike}
\small\bibliography{aa_bib}{}

\section*{Addendum}
As pointed out by Don McLeish and Zhenyu Cui in a personal communication (November 2012), the definition of $\tau_i$ in the proof of Theorem~\ref{C abstract} leads to an error in its proof; precisely, the convergence in  \eqref{E RV lim} on the set $H$ cannot be guaranteed. By slightly changing  the definition of $\tau_i$ we will correct this error here. 

In the notation of that proof, define the sequence of stopping times $\{\widetilde{\tau}_i\}_{i \in \N}$ as
        \begin{align*} 
        \widetilde{\tau}_i := \inf \left\{t \geq 0: f\left(\widetilde{L}_t + \left\langle \widetilde{L} \right\rangle_t,  \left\langle \widetilde{L} \right\rangle_t\right)  \exp\left(-\widetilde{L}_t - \frac{\left\langle \widetilde{L} \right\rangle_t}{2}\right)  \geq i\right \}. 
    \end{align*}
 Observe next that there exists $j_i \in \N$  for all $i \in \N$  such that
 \begin{align*}
 	\Q\left(H \cap \left\{\widetilde{\tau}_i > \frac{j_i-1}{j_i}\right\}\right) \leq \frac{\Q(H)}{2^{i+1}}
 \end{align*}
    due to the continuity of probability measures and the fact that $H \subset \{\tau_i<1\}$; this inclusion holds since $\widetilde{L}$ can be represented as a time-changed Brownian motion, as in the proof of  Theorem~\ref{C abstract}. Now, modify the definition of $\tau_i$ and define $\tau_i$ as ${\tau}_i := \widetilde{\tau}_i \wedge (j_i-1)/j_i$.  Then $\tau_i \in \mathcal{T}$ and  \eqref{E RV lim} holds on the set $\widetilde{H} := \cap_{i = 1}^\infty \{\tau_i = \widetilde{\tau}_i\} \subset H$. Since $\Q(\widetilde{H})>0$, we can now conclude as in the proof of Theorem~\ref{C abstract}.

\end{document}